\documentclass[12pt,thmsa]{amsart}
\usepackage{amsmath}
\usepackage{amsfonts}
\usepackage{amssymb}
\usepackage{graphics,xcolor,enumerate}

 \usepackage[utf8]{inputenc}
\usepackage{parskip}
\usepackage{fancybox} 
\usepackage{graphicx}

\usepackage[USenglish]{babel}

\newcommand{\Et}{\mathbb{E}^3}
\newcommand{\Ed}{\mathbb{E}^{d}}

\newcommand{\inte}{\operatorname*{int}}

\newcommand{\aff}{\operatorname*{aff}}

\newcommand{\bd}{\operatorname*{bd}}

\newtheorem{proposition}{Proposition}
\newtheorem{lemma}{Lemma}
\newtheorem{theorem}{Theorem}

\newtheorem{remark}{Remark}
\newtheorem{definition}{Definition}

\newtheorem{conjecture}{Conjecture}

\title[Generalized bodies of revolution]{Characterizations of generalized convex bodies of revolution}

\begin{document}

 \author[Alfonseca]{
M. Angeles Alfonseca}
\address{North Dakota State University, USA} \email{maria.alfonseca@ndsu.edu} 

\author[Cordier]{M. Cordier}    
\address{Chatham University, USA}
\email{m.doyle@chatham.edu}

\author[Morales]{E. Morales-Amaya}
\address{Facultad de Matem\'aticas-Acapulco\\
Universidad Aut\'onoma de Guerrero, Mexico}
\email{emoralesamaya@gmail.com}

\author[Verdusco]{D. J. Verdusco Hern\'andez}
\address{Facultad de Matem\'aticas-Acapulco\\
Universidad Aut\'onoma de Guerrero, Mexico}
\email{diana.janett.h@gmail.com}

 \thanks{The first author is supported in part by the Simons Foundation gift MPS-TSM-00711907. The second author is supported in part by the AMS-Simons grant ASPUI-24-CRMCXU. SNI-Conahcyt of Mexico, grant 21120, supports the third author.}


\begin{abstract}
In this work we prove that either a sequence of axes of symmetry or a sequence of hyperplanes of symmetry of a convex body $K$ in the Euclidean space $\Ed, d\geq 3$, are enough to guarantee that $K$ is a generalized body of revolution (and in some cases a sphere).
\end{abstract}

\maketitle

\section{Introduction}

Let $K\subset \Ed$, $d\geq 3$, be a convex body. In dimension 3, an \textit{axis of symmetry} of $K$ is a line $L$ such that $K$ remains invariant after a rotation by the angle $\pi$ around $L$. Naturally, the notion axis of symmetry of a convex body is easily generalizable to higher dimensions (see Section \ref{defs}).

If the body $K \subset \Ed$ satisfies one of the following two properties:  
\begin{enumerate}[(A)] 
\item all the lines $L$ passing through a point $p$ are axes of symmetry of $K$, 
\item all the hyperplanes passing through a point $p$ are hyperplanes of symmetry of $K$, 
\end{enumerate}
then $K$ is a sphere with center at $p$. Indeed, property (A) implies that all the sections of $K$ are centrally symmetric, and it is well known that in this case $K$ is an ellipsoid (see for example \cite{bg,la,mm1}). On the other hand, the only ellipsoid such that all lines through a given point are axes of symmetry is the sphere \cite{Au}.
As for property (B), if $K$ is symmetric with respect to a hyperplane $H$ passing through $p$, then for all segments contained in $K$ perpendicular to $H$, their midpoints lie on $H$. By Brunn's theorem \cite{br} (see also \cite[pg 44]{mmo}), $K$ is an ellipsoid. However, the only ellipsoid such that all hyperplanes  through a given point are hyperplanes of symmetry is a sphere \cite{AM}.

It is natural to try to reduce the quantity of axes of symmetry or hyperplanes of symmetry needed to guarantee that a convex body $K\subset \Ed$ is a body of revolution (in some cases, a sphere).  For example,  it was proven in \cite{jemomo} that if there exists a plane $H$ in $\Et$ containing a point $p$ such that every line $L\subset H$ through $p$ is an axis of symmetry of $K$, then $K$ is a body of revolution.

In this work we generalize the aforesaid result to general dimension $d$, when $K$ has enough $k$-axes of symmetry, obtaining that $K$ is a $k$-body of revolution (see Section \ref{defs}, Definitions \ref{kaxes} and \ref{chupadota}.) Our first two results are as follows.

\begin{theorem}\label{grandota}
Let $K\subset \Ed$ be a convex body, $d\geq 3$, and let $k$ be a positive integer, $2\leq k\leq d-1$. Assume that there exist a $k$-flat $\Lambda$ and a point $p\in \Lambda$,  such that every line $L$ passing through $p$ and contained in $\Lambda$ is an axis of symmetry of $K$. Then $K$ is a $k$-body of revolution.
\end{theorem}


\begin{theorem}\label{copaoro}
Let $K\subset \Ed$ be a  convex body, $d\geq 3$, and let $k$ be an integer $1\leq k \leq d-2$. Assume that there exists a $(k+1)$-flat $\Lambda$ and a point $p\in \Lambda$, such that every $k$-flat $\Gamma \subset \Lambda$, $p\in \Gamma$, is a $k$-axis of symmetry of $K$. Then $K$ is a  $(k+1)$-body of revolution.
\end{theorem}

 We also show either a sequence of axes of symmetry or a sequence of hyperplanes of symmetry are enough to guarantee that $K$ is either a $k$-body of revolution or a sphere. 

\begin{theorem}\label{dream}
Let $K\subset \Ed$ be a convex body, $d\geq 3$. Suppose that there exists a  sequence $\{L_n\}\subset \Ed $ of axes of symmetry of $K$ such that:
\begin{itemize}
\item [(i)]  The sequence contains infinitely many distinct lines.
\item [(ii)] $\dim(\aff\{L_1,L_2,\ldots\}) =k$, $2\leq k \leq d$,
\item [(iii)]  If $i\not=j$, $L_i$ is not perpendicular to $L_j$.
\end{itemize}
Then 
$K$ is a $k$-body of revolution. Moreover, if $k=d$, then $K$ is a sphere.
 \end{theorem}

\begin{theorem}\label{fantasia}
Let $K\subset \Ed$ be a convex body, $d\geq 3$, and let $\{\Pi_i\}$ be a sequence of distinct hyperplanes in 
$\Ed$. Assume that for $i=1,2,\ldots$, $\Pi_i$ is a hyperplane of symmetry of $K$, and there exists an integer $k$, $1\leq k \leq d-2$, and a $k$-flat $\Gamma$ such that 
\begin{eqnarray}\label{miami}
\Gamma \subset \Pi_i  \textrm{ }\textrm{ } \textrm{for} \textrm{ } i=1,2,\ldots,
\end{eqnarray}
 and $k$ is maximal with this property.  
 Then $K$ is a $(d-k)$-body of revolution.
\end{theorem}

In Theorem \ref{fantasia} we assume that all the hyperplanes $\Pi_i$ contain the same subspace $\Gamma$. In a  related result  obtained in \cite[Lemma 18]{barthe} in connection to  Mahler's conjecture,  the authors explicitly avoid the case where the family of hyperplanes of symmetry of the convex body share a subspace (this is the irreducibility hypothesis in their Lemma), and conclude that $K$ must be a sphere. Thus, combining their result with our Theorem \ref{fantasia}, we have a  complete  characterization of convex bodies in terms of a sequence of hyperplanes of symmetry. 

In our last result, instead of axes of symmetry we consider $k$-axes of rotation. 

\begin{theorem}\label{brasil}
Let $K\subset \Ed$ be a convex body, $d\geq 3$, let $p\in \Ed$ be a point and let $k$ a positive integer, $k\geq 3$. Su\-ppo\-se that every $(d-2)$-flat passing through $p$ is a $(d-2)$-axis of rotation of order 
$k$ of $K$, then $K$ is a sphere with center at $p$.
\end{theorem}


The three dimensional case of Theorem \ref{brasil} follows immediately from the fact that the classification of the finite subgroups of $SO(3)$ is well understood 
(see for example  \cite[ Thm. 19.2]{arm}). Our proof is in all dimensions and it involves the S\"uss-Schneider's cha\-rac\-te\-ri\-za\-tion of the sphere in terms of concurrent congruent sections (see \cite{schneider, suss}).

 \section{Definitions and notation}\label{defs}

 Let $\mathbb{E}^{d}$ be the Euclidean space of dimension $d$ endowed with the usual scalar product $\langle \cdot, \cdot\rangle : \mathbb{E}^{d} \times \mathbb{E}^{d} \rightarrow \mathbb{R}$. 
 Let $\mathbb{S}^{d-1}=\{x\in \mathbb{E}^{d}: ||x|| = 1\}$ be the unit sphere in $\mathbb{E}^{d}$.
 
 Given a set $A\subset \Ed$,  we will denote the affine hull of $A$ by $\aff\{A\}$. We say that $\Gamma\subset \Ed$  is a $k$-\textit{flat}, $1\leq k\leq d-1$, if there exits a $k$-dimensional subspace $\Lambda\subset \Ed$ and a point $p\in \Ed$ such that $\Gamma=p+\Lambda$. For a $k$-flat $\Lambda\subset \Ed$ and a point $p\in  \Lambda$, we denote by
$\Lambda ^\perp$ the $(d-k)$-flat orthogonal to $\Lambda$ and passing through $p$. A $(d-1)$-flat will be called a hyperplane.

A {\it convex body} $K\subset {\mathbb E^d}$, $d\ge 2$,   is a convex compact set with non-empty interior. A \textit{convex hypersurface} is the boundary of a convex body $K$ in $\Ed$ and it will be denoted by $\bd K$. We will denote by $\inte K$ the set $K\backslash \bd K$.

Let $L\subset \Et$ be a  line and let $\theta \in [0,2\pi]$.  We denote by $R_{(L,\theta)}:\Et \rightarrow \Et$ the rotation by the angle $\theta$ around the axis $L$. To simplify the notation, the maps $R_{(L,\pi)}$, $R_{(L,\frac{2\pi}{n})}$ ($n \geq 3$) will be denoted, respectively, by  $R_L$ and $R_{L,n}$. A line $L\subset \Et$ is said to be an {\it axis of symmetry of order $n$, $n \geq 2$} (or an $n$-axis of symmetry)  if $R_{L, n}(K)=K$. In the case $n=2$, a 2-axis of symmetry of the convex body $K$ will just be called \textit{axis of symmetry} of $K$. 


Let $K \subset \Ed$ be a convex body, $d\geq 3$ and let $\Pi$ be a  hyperplane. We denote by $S_{\Pi}: \Ed \rightarrow \Ed$ the reflection with respect to $\Pi$. The hyperplane $\Pi$ is said to be a \textit{hyperplane of symmetry} of $K$ if 
\[S_{\Pi}(K)=K.
\]

Observe that if $M\subset \Et$ is a body of revolution with axis $L$ and with a hyperplane of symmetry $\Pi$ perpendicular to $L$, then \textit{every line $\Gamma$ contained in $\Pi$ and passing through the point $\Pi \cap L$ is an axis of symmetry of} $K$. Indeed, every such line $\Gamma$ is contained in two perpendicular planes of symmetry of $K$. In Theorem \ref{grandota} we prove, for every dimension $d$, $d\geq 3$, the converse of the aforesaid statement. A three-dimensional proof of this result appeared in \cite[Thm 7]{ACJM} and was used in \cite{jemomo} (also in dimension 3) to prove that a body with a \textit{false axis of revolution} is a sphere.

Next, we generalize the definitions of axis of rotation, plane of symmetry and axis of symmetry to $k$-dimensional flats. 

\begin{definition}\label{d2ax}\textbf{ $\mathbf{(d-2)}$-axis of rotation.}
Let $K \subset \Ed$ be a convex body, $d\geq 3$, and let $\Gamma$ be a $(d-2)$-flat. 
Denote by $R$ the map $S_{\Pi_2}\circ S_{\Pi_1}$, where $\Pi_1,\Pi_2$ are two $(d-1)$-flats containing $\Gamma$, whose unit normal vectors $u_1,u_2$ make an angle of $\pi/n$. Then $\Gamma$ is said to be a $(d-2)$-\textit{axis of rotation of order $n$} of the body $K$ if 
\begin{eqnarray}\label{cigala}
K= R(K).
\end{eqnarray}
\end{definition}
The map $R$ can be considered as a \textit{rotation} with \textit{axis} $\Gamma$ by an angle $2\pi/n$ (see, for example, \cite{conway} Pgs. 6 and 23). It is clear that $R^n=I$.

\begin{definition}\label{axsym}\textbf{Axis of symmetry.} Let $K \subset \Ed$ be a convex body, $d\geq 3$, and let $L \subset \Ed$ be a line. Let $O(d)$ be the orthogonal group in $\Ed$. We denote by $R_{L}\in O(d)$ the orthogonal transformation satisfying $R_L(L)=L$, and $R_L|_{L^\perp}=-I|_{L^\perp}$.
The line $L$ is said to be a \textit{axis of symmetry of} $K$ if $R_{L}(K)=K$.
\end{definition}

If $L$ is an axis of symmetry of $K$,  all the sections of  $K$ by hyperplanes orthogonal to $L$ are centrally symmetric with center on $L$. On the other hand, all the sections of $K$ by hyperplanes containing $L$ have $L$ as a line of symmetry. This observation motivates the following $k$-dimensional generalization of axis of symmetry of a convex body.

\begin{definition}\label{kaxes}
Let $K \subset \Ed$, $d\geq 3$, be a convex body and let $\Gamma$ be a $k$-flat, $1\leq k \leq d-2$.
\begin{enumerate}[(a)]

\item We say that $\Gamma$ is a $k$-\textit{plane of symmetry} of $K$ if $\Gamma$ is contained in a $(k+1)$-flat $\Sigma$ such that the section $K \cap \Sigma$ has $\Gamma$ as hyperplane of symmetry. 

\item  We say that $\Gamma$ is  a $k$-\textit{axis of symmetry} of $K$ if for every $(k+1)$-flat $\Lambda$ containing $\Gamma$, the section $\Lambda \cap K$ has $\Gamma$ as a hyperplane of symmetry.  In other words,  $\Gamma$ is  a $k$-axis of symmetry if it is 
a $k$-plane of symmetry for every $(k+1)$-flat $\Lambda$ containing $\Gamma$.

\end{enumerate}
\end{definition}

As examples, consider the cross-polytope $P$ in $\Ed$, and the double cone $C$ in $\Ed$ obtained by rotating the triangle with vertices $(1,0,0, \ldots,0)$ and $(0,0,\ldots,\pm 1)$ around the $x_d$-axis. Then the plane $x_1 x_d$ is a 2-plane of symmetry of $P$, and a 2-axis of symmetry of $C$.


\begin{definition}\label{chupadota} 
Let $K\subset \mathbb{E}^{d}$ be a convex body, $d\geq 3$,  and let $k$ be an integer, $1\leq k<d$. We say that $K$ is a $k$-\textit{body of revolution} if there exists a decomposition of 
$\mathbb{E}^{d}$ in the form 
$\mathbb{E}^{d}=\mathbb{E}^{d-k} \bigoplus \mathbb{E}^{k}$, $\mathbb{E}^{d-k}$ orthogonal to 
$\mathbb{E}^{k}$, such that for every affine $k$-flat 
$\Gamma$ parallel to $\mathbb{E}^k$, the set $\bd K\cap \Gamma$ is a $k$-dimensional Euclidean ball with center in $\mathbb{E}^{d-k}$.
\end{definition}

We observe that a $(d-1)$-body of revolution is just a usual body of revolution around a 1-dimensional axis, while a $1$-body of revolution is a convex body which is invariant under reflection on the hyperplane $\mathbb{E}^{d-1}$.


\section{Auxiliary results}

The proof of Theorem \ref{dream} uses the notion of $n$-star of line, and several facts about sequences of lines or subspaces of symmetry, which we present in this section.

\begin{definition}
A family of lines $\{L_1,\ldots, L_n\}$ is called  an $n$-star of lines with apex $x_0$ if the lines $L_i$ are in a plane, are concurrent at $x_0$, and the angle between two consecutive lines is $\frac{2\pi}{n}$.  
\end{definition}
\begin{definition}
Let $L_1$ and $L_2$ be two axes of symmetry of the convex body $K$. The star determined by $L_1$ and $L_2$, which will be denoted by $\Sigma (L_1,L_2)$, is the family of lines  
$\{T_n\}$ constructed in the following way: $T_1=L_1, T_2=L_2$, and, in general,
\[
 T_k=R_{T_{k-1}}(T_{k-2}).
\]
(The map $R_{L}$, for a line $L\subset \Ed$, was defined in Section 2).
  \end{definition}
We observe that, on one hand, each line in the family $\{T_n\}$ is an axis of symmetry of $K$; on the other hand, $T_i\subset \aff\{T_1,T_2\}$ for all $i$. 
If $L_1$ and $L_2$ are two lines with non-empty intersection, we denote by $\Omega (L_1,L_2)$  the set of all lines contained in the plane $\aff \{L_1,L_2\}$ and passing through the point  
$L_1\cap L_2$. 
\begin{remark}\label{escandalo}
Depending on the angle between $L_1$ and $L_2$ either  $\Sigma (L_1,L_2)$ is an  
$n$-star of lines for some integer $n$ or $\Sigma (L_1,L_2)$ is a dense set in 
$\Omega (L_1,L_2)$.
\end{remark}

Next, we define a metric in the set of subsets of $\Ed$. For a  non-empty set $A$, the distance from a point $p$ to $A$ is denoted by $pA$, and defined as the greatest lower bound of $px$, where $x$ is a point in $A$. 

We select a point $p \in \Ed$ and define the distance of the sets $M,N$ as 
\begin{eqnarray}\label{edorron}
\delta_p(M,N)=\sup_{x\in \Ed}|  xM-xN |e^{-px}.
\end{eqnarray}
With such distance the family of non-empty closed sets of $\Ed$ is a metric space (see (3.8) of \cite{bu}). Furthermore, this metric space is a finitely compact space (Theorem (3.15) of \cite{bu}) and this  property will play a fundamental role our proofs. 
When we speak about convergence of a sequence of subsets of $\Ed$, we mean  that the metric $\delta_p$, given by (\ref{edorron}), is involved, even if not stated explicitly.

The next Lemma lists several auxiliary results. We will prove III. and leave the remaining proofs to the reader. 

\begin{lemma}\label{contento}

\begin{itemize}

$~$

\item[I.] Let $\Phi$ be a plane convex figure and let   $\{L_1,\ldots, L_n \}$ be the collection of all its lines 
of symmetry. Then,  $\{L_1,\ldots,L_n \}$ is an $n$-star of lines.
 
\item[II.] Let $K \subset \Ed$ be a convex body and let $\{H_i \}$ be a sequence of hyperplanes that intersect $\inte K$. Suppose that $H_i \rightarrow H$, $L_i \subset H_i$ is a $(d-2)$-plane of symmetry ($p_i \in H_i$ is a center of symmetry) of $H_i \cap K$, and $L_i \rightarrow  L$ ($p_i \rightarrow  p$); then, $L$ is a $(d-2)$-plane of symmetry ($p$ is center of symmetry) of $H \cap K$. 

\item[III.] Let $K\subset \Ed, d\geq3,$ be a convex body. Suppose that $\{L_n\} \subset \Ed$ is a sequence of axes (hyperplanes) of symmetry of $K$ and $L$ is a line (hyperplane) such that $L_n \rightarrow L$. Then $L$ is an axis (hyperplane) of symmetry of $K$.

\item[IV.] Let $\{K_i\}$ be a sequence of convex figures such that  $K_i \rightarrow K$ and, for every $i\in\mathbb{N}$, the figure $K_i$ has two lines of symmetry determining an angle $\theta_i$. If 
$\lim_{i\to\infty}\theta_i=0$, then  $K$ is a circle.
\end{itemize}

\end{lemma}

\begin{proof}

[III.]  Since $L_n \rightarrow L$, by Theorem 1.8.7 in \cite{sch}, for all $q\in L \cap K$, there exists $q_n \in L_n \cap K$ such that $q_n \rightarrow q$. We denote by $\Gamma_n$  the orthogonal hyperplane to $L_n$ passing through $q_n$, and by $\Gamma$ the hyperplane orthogonal to $L$ passing through $q$. Since $L_n$ is an axis of symmetry of  $K$, $\Gamma_n \cap K$ is centrally symmetric with center at $q_n$. In virtue that $L_n \rightarrow L$ and $q_n \rightarrow q$, we have $\Gamma_n \rightarrow \Gamma$. Thus $\Gamma_n \cap K \rightarrow \Gamma \cap K$. From II. it follows that $\Gamma \cap K$ is centrally symmetric with center at $q$. Thus $L$ is an axis of symmetry of $K$.
\end{proof} 

The following result is well known. However, a proof in the case of a set with an infinite number of symmetries is not easily found in the literature (\cite[Sec.3.2 Thm 5]{conway}  contains a proof for the finite case). We include here an elementary proof in terms of the uniqueness of the circumsphere of a convex body.

 \begin{lemma}\label{concurrentes}
Let $K\subset \Ed$ be a convex body. We denote by $\Omega$ the circumsphere of $K$ and by $o$ its center. Then any hyperplane of symmetry and any axis of symmetry of $K$ must contain the point $o$.
\end{lemma}
\begin{proof}
Suppose that there exists a hyperplane  of symmetry $\Gamma$ of $K$ such that $o\notin \Gamma$. Then $K=S_{\Gamma}(K)$. Furthermore, since $S_{\Gamma}(\Omega)\not= \Omega$ and $K\subset \Omega$, we have  $K=S_{\Gamma}(K)\subset S_{\Gamma}(\Omega)$. Hence, $K\subset \Omega \cap S_{\Gamma}(\Omega)$. This implies that there exists a sphere of smaller radius that contains $K$, contradicting the assumption that $\Omega$ is the circumsphere of $K$.

On the other hand, let $L$ be an axis of symmetry of $K$, and assume that $L$ does not contain $o$. Then, 
$\Omega\not=R_{L}(\Omega)$. Since $K\subset \Omega$ we have that 
\[
K=R_{L}(K)\subset R_{L}(\Omega),
\]
because $L$ is an axis of symmetry of $K$. Thus, $K\subset \Omega \cap R_{L}(\Omega)$. Consequently, $K$ is contained in a sphere of smaller radius than $\Omega$, which again is a contradiction. 
\end {proof}

The following result was proven in \cite{barthe} in the context of Mahler's conjecture. We will use it in our proof of Theorem \ref{fantasia}.

\begin{proposition}\label{frances}Let $K\subset \Ed$ be a convex body, $d\geq 3$, and let $\{\Pi_i\}$ be a sequence of hyperplanes in $\Ed$. Suppose that $\Pi_i$ is a hyperplane of symmetry of $K$ for $i=1,2,\ldots$. If, in addition, there does not exist an integer $k$, $1\leq k \leq d-2$, and a $k$-flat $\Gamma$ such that $ \Gamma \subset  \Pi_i$ for $i=1,2,\ldots$, then $K$ is an $(d-1)$-sphere. 
\end{proposition}

\section{Proof of Theorem \ref{grandota}}
First, we consider the case $k=d-1$. We claim that $K$ is centrally symmetric. Let  $\Lambda^{\perp}$ denote the line orthogonal to $ \Lambda$ that passes through $p$. For every  hyperplane $\Omega$, $  \Lambda^{\perp} \subset \Omega$, we consider the line $ L\subset   \Lambda$ orthogonal to $\Omega$ and passing through $p$. Since $L$ is an axis of symmetry of the body $K$, by hypothesis, it follows that $\Omega \cap K$ is centrally symmetric with center at $p$. Varying $\Omega$, $  \Lambda^{\perp} \subset \Omega$, we obtain that $K$ is centrally symmetric with center at $p$.

We take a coordinate system such that $p$ is the origin. Let $\Omega$ be an hyperplane with $\Lambda ^{\perp} \subset \Omega$. We will show that every such $\Omega$ is an hyperplane of symmetry of $K$. Let $L\subset  \Lambda$ be the line passing through $p$ and perpendicular to $\Omega$. Let $\Phi$ be an affine hyperplane parallel to $\Omega$ such that $ \Phi \cap K\not= \emptyset$. We take $x\in \Phi \cap \bd K$. We claim that there exists a line $\Delta$ parallel to $L$, such that $x\in \Delta$ and $\Delta \cap \bd K=\{x,y\}$ where $y\in (-\Phi \cap \bd K)$. Indeed, since $L$ is an axis of symmetry of the body $K$, the sections $\Phi \cap \bd K$ and $-(\Phi \cap \bd K)$ are centrally symmetric with centers $c, -c$, respectively, where $c,-c\in L$. We observe that
\begin{eqnarray}\label{pastelito}
-(\Phi \cap \bd K)=(\Phi \cap \bd K)-2c.    
\end{eqnarray}
This holds because  $-c+(c-x)=-x\in -(\Phi \cap \bd K)$, and thus  $-c-(c-x)=x-2c \in -(\Phi \cap \bd K)$). Hence, if we define $\Delta=x+L$ and $y=x-2c$, by (\ref{pastelito}) we have that $y\in -(\Phi \cap \bd K)$. We have shown that $\Omega$ is an hyperplane of symmetry of $K$.

 Now let $\Gamma$ be a hyperplane parallel to $\Lambda$ such that $\Gamma \cap K\not=\emptyset$. By the argument in the previous paragraph, the section $\Gamma \cap K$ has the property that every $(n-2)$-plane $H$ passing through $\Gamma \cap \Lambda^{\perp}$ is a $(n-2)$-plane of symmetry of $\Gamma \cap K$ (every such plane $H$ can be expressed in the form $H:=\Omega \cap \Gamma$, with $\Omega$ containing $\Lambda^{\perp}$). Hence, property (B) from the introduction,  $\Gamma \cap K$ is an sphere with center at $\Gamma \cap \Lambda^{\perp}$. We have shown that $K$ is a $(d-1)$-body of revolution.

We consider now the case $2 \leq  k \leq d-2$. We are going to prove that if $\Lambda'$ is a $k$-flat parallel to $\Lambda$ and such that $\Lambda' \cap \inte K \not= \emptyset$, then $\Lambda' \cap K$ is a ball with center at $\Lambda ^\perp$. In other words, $K$ is a $k$-body of revolution.

We denote by $\Delta$ the $(k+1)$-flat spanned by $\Lambda$ and $\Lambda'$, i.e., $\Delta=\aff\{\Lambda,\Lambda'\}$. Since every line $L$, with $p\in L$ and $L\subset \Lambda$, is an axis of symmetry of $K$, then every such $L$ is an axis of symmetry of the body $\Delta \cap K$. Thus, in virtue of case $k=d-1$ of Theorem \ref{grandota}, $\Delta \cap K$ is $k$-body of revolution, i.e., for every $k$-flat $\Gamma$, parallel to $\Lambda$, the section $\Gamma \cap (\Delta\cap K)$ is a $(k-1)$-sphere with center at the line $\Lambda ^\perp \cap \Delta$. Therefore, in particular, $\Lambda' \cap \inte K$ is a $(k-1)$-sphere with center at $\Lambda ^\perp$ as we have claimed.  $\Box$ 

\section{Proof of Theorem \ref{copaoro}}  




\begin{lemma}\label{mundial}
Let $K\subset \Ed$ be a convex body, $d\geq 3$, and let $\Gamma$ be a $k$-flat, $1\leq k \leq d-2$. Then $\Gamma$ is a $k$-axis of symmetry of  $K$ if and only if, for all $(d-k)$-flat $\Omega$ orthogonal  to $\Gamma$ and such that $\Omega \cap K \not=\emptyset$, the section $\Omega \cap K$ is centrally symmetric with center at the point $\Omega \cap \Gamma$.
\end{lemma}  
\begin{proof}
Let $\Gamma$ be a $k$-axis and let $\Omega$ be a $(d-k)$-flat such that $\Omega$ is orthogonal to $\Gamma$ and $\Omega \cap K \not=\emptyset$. For every $(k+1)$-flat $\Pi, \Gamma \subset \Pi$,  the section $\Pi \cap K$ is symmetric with respect to $\Gamma$. Hence the line segment $\Omega \cap (\Pi \cap K)$ has $\Omega \cap \Gamma$ as its midpoint. Varying $\Pi$ and $\Gamma \cap \Pi$, it follows that $\Omega \cap K$ is centrally symmetric with center at $\Omega \cap \Gamma$. 
\end{proof}

\textit{Proof of Theorem} \ref{copaoro}. We will show that 
all the $(k+1)$ sections of $K$ parallel to $\Lambda$ are Euclidean balls, whose centers are situated in a $(d-(k+1))$-flat orthogonal to 
$\Lambda$ passing through $p$, i.e., $K$ is a $(k+1)$-body of revolution.

First, we prove that $K$ is centrally symmetric. 
Recall that $\Lambda$ is a $(k+1)$-flat, and thus $\Lambda^{\perp}$ is $(d-(k+1))$-flat.
For every $(d-k)$-flat $\Omega$, $\Lambda^{\perp} \subset \Omega$, we consider the $k$-flat 
$\Omega^{\perp}\subset \Lambda$. By hypothesis, $\Omega^{\perp}$ is a $k$-axis of symmetry of $K$. From Lemma \ref{mundial}, it follows that $\Omega \cap K$ is centrally symmetric with center at $p$. Varying $\Omega$, $ \Lambda^{\perp} \subset \Omega$, we obtain that $K$ is centrally symmetric with center at $p$.    

We take a system of coordinates with $p$ as the origin. We will show next that every $(d-k)$-flat $\Omega$, $\Lambda^{\perp} \subset \Omega$, is a $(d-k)$-axis of symmetry of $K$. Observe that this fact is, in a sense, dual to our hypothesis.

Let $\Omega$ be a $(d-k)$-flat such that $\Lambda^{\perp} \subset \Omega$, and let $\Gamma\subset \Lambda$ be a $k$-flat passing through $p\in \Gamma  $ and orthogonal to $\Omega$. 
We need to show that for every $(d-k+1)$-flat $\Sigma$ containing $\Omega$, $\Omega$ is a $(d-k)$-plane of symmetry of $\Sigma \cap K$. Let $\Sigma$ be a $(d-k+1)$-flat with   $\Omega \subset \Sigma$ and let
$\Phi\subset \Sigma$ be a $(d-k)$-flat parallel to $\Omega$ (thus, orthogonal to $\Gamma$) such that $ \Phi \cap K\not= \emptyset$. We take $x\in \Phi \cap \bd K$. We claim that there exists a line $\Delta$ perpendicular to  $\Omega$, with $x\in \Delta$ and $\Delta \cap \bd K=\{x,y\}$, where $y\in (-\Phi) \cap \bd K$, and $\Delta \subset \Sigma$. By Lemma \ref{mundial}, the sections $\Phi \cap \bd K$ and $-(\Phi \cap \bd K)$ are centrally symmetric. We denote their centers by $c, -c$, respectively. Notice that $c,-c\in \Gamma$ and
\begin{eqnarray}\label{rica}
-(\Phi \cap \bd K)=(\Phi \cap \bd K)-2c.    
\end{eqnarray}
(Since $-c+(c-x)=-x\in -(\Phi \cap \bd K)$, we get $-c-(c-x)=x-2c \in -(\Phi \cap \bd K)$). Thus, letting $L(p,c)$ be the line passing through the points $p,c$, and  defining $\Delta=x+L(p,c)$ and $y=x-2c$, it follows from (\ref{rica}) that $y\in -(\Phi \cap \bd K)$, and we obtain what we have claimed.

Now, let $\Lambda'$ be a $(k+1)$-plane parallel to $\Lambda$ such that $\Lambda' \cap K\not=\emptyset$. 
Consider a line $L\subset \Lambda'$  passing through 
$\Lambda' \cap \Lambda^{\perp}$ and let $\Omega$ be a $(d-k)$-flat containing $\Lambda^{\perp}$ such that $L=\Omega \cap \Lambda'$. In order to prove that $L$ is a line of symmetry of $\Lambda' \cap K$,  we will show that for every two dimensional plane $\Pi\subset \Lambda'$ containing $L$ the section $\Pi \cap K$ has $L$ as a line of symmetry. Let $\Pi \subset \Lambda'$ be a plane with $L\subset \Pi$, and let $M\subset \Pi$ be a line perpendicular to $L$. Denote by $\Sigma$ the $(d-k+1)$-plane defined by $\Omega$ and $M$. By virtue that $\Omega$ is a $(d-k)$-axis of symmetry of $K$,    
$\Omega$ is a $(d-k)$-plane of symmetry of $\Sigma \cap K$.
It follows from the choice of $M$ that $\Pi \cap K$ has $L$ as line of symmetry (notice that $M\perp \Omega$ because $M\subset \Pi \subset \Lambda'$ and $\Lambda^{\perp}\subset \Omega$). From the property (A) mentioned in the introduction, we conclude that $\Lambda' \cap K$ is a Euclidean ball with center at $\Lambda' \cap \Lambda^{\perp}$.   $\Box$ 

\section{Proof of Theorem 3}

\begin{lemma}\label{kiara} 
Let $K\subset \Ed, d\geq 3$, be a convex body. Let $\Pi$ be a $k$-dimensional subspace, $2\leq k\leq d-1$, and let $L'\subset \Ed \backslash \Pi$ be a line through the origin $o$, such that  $L'\not=\Pi^{\perp}$. 
Suppose that every line $L\subset \Pi$ through $o$ is an axis of symmetry of $K$ and, furthermore, $L'$ is an axis of symmetry of the body $K$. Then every line through $o$ and contained in $\aff\{\Pi, L'\}$ is an axis of symmetry of $K$.
\end{lemma}
\begin{proof}
Let $v$ be a unit vector parallel to $L'$. For each $u\in \mathbb{S}^{k-1} \subset \Pi$, we denote by $\alpha(u)$ the angle between the vector $u$ and $v$. Let $\delta_1$ and $\delta_2$ be the minimum and the maximum of $\alpha (\cdot)$.  Since $L'\subset \Ed \backslash \Pi$ and $L'\not=\Pi^{\perp}$ it follows that $\delta_1>0$ and $\delta_2<\pi/2$. By virtue of the continuity of $\alpha(u)$ and the compactness of $\mathbb{S}^{k-1}$, we have that 
$\alpha (\mathbb{S}^{k-1})=[\delta_1,\delta_2]$. We denote by $F$ the subset of irrational multiples of $\pi$ in $[\delta_1,\delta_2]$ and by $E$ the set $\alpha^{-1}(F)$. Since $F$ is dense in $[\delta_1,\delta_2]$, the set $E$ is dense in $\mathbb{S}^{k-1}$. For each $u\in E$, we denote by $L(u)$ the line through the origin with direction $u$, and by $\Pi(u)$ the plane 
$\aff \{u,v\}$. Since $u\in E$, we have that $\alpha(u)\in F$ and, consequently, 
$\Sigma(L(u), L')$ is a set of axes of symmetry of $K$ in $\Pi(u)$ which is dense in 
$\Omega(L(u), L')$. Thus, by Lemma \ref{contento} [III], every line in $\Pi(u)$ through $o$, is an axis of symmetry of $K$.

Let $w$ be a unit vector in span$\{\Pi,L'\}$. We will show that  there exists an axis of symmetry $L_w$ of $K$ such that $L_w=span\{w\}$. Let $e \in \mathbb{S}^{k-1}$ such that 
$span\{e\}=\Pi \cap span\{v,w\}$. Since $E$ is dense in $\mathbb{S}^{k-1}$, there exists a sequence $\{e_i\}\subset E$ such that $e_i \rightarrow e$, when $i \rightarrow \infty$. It follows that $\Pi(e_i) \rightarrow span\{v,w\}$. Hence, we can find a sequence of lines $\{\Delta_i\}$, $\Delta_i \subset \Pi(e_i)$, such that $\Delta_i \rightarrow L_w=span\{w\}\subset span\{v,w\}$. By the argument in the previous paragraph, every line $\Delta_i$ is an axis of symmetry of the body $K$. Therefore, by Lemma \ref{contento} [III], $L_w$ is an axis of symmetry of $K$.
\end{proof}


\textit{Proof of Theorem} \ref{dream}. In virtue of Lemma \ref{concurrentes}, there exists a point $o\in \Ed$ such that  $o\in L_n$, for all $n$. We take a coordinate system having $o$ as the origin. We assume that $V=span\{L_1,L_2,\ldots\}$ has dimension $k$, $k\geq 2$.  

First, we consider the case when there exists a two-dimensional plane $H\subset V$ and an infinite subsequence $L_{n_i}$ of $\{L_n\}$ such that $\{L_{n_i}\} \subset H$, for all $i$. By the compactness 
of $\mathbb{S}^{2}$, there is an infinite subsequence of $L_{n_i}$, which will be denote again by $L_{n_i}$, and a line $L\subset H$ through $o$ such that $L_{n_i} \rightarrow L$, when $i \rightarrow \infty$. By Lemma \ref{contento} [III], $L$ is an axis of symmetry of $K$.

We will show  that every line $\Delta \subset H$,  $o \in \Delta$, is an axis of symmetry of  $K$. By Lemma \ref{contento} [III], it is enough to find a sequence of axes of symmetry $\{\Delta_n\}\subset H$ of $K$ such that $\Delta_n \rightarrow \Delta$.  First of all,  suppose that, for all $i$, $\Sigma (L,L_{n_i})$ is a $f(n_i)$-star of lines for some integer $f(n_i)$. Since  $L_{n_i} \rightarrow L$, it follows  that $f(n_i) \rightarrow \infty$. Thus, given $\epsilon > 0$, there exists an integer $n_i$ such that  the difference between the angle of $\Delta$ and the angle the some line $\Delta_{k(n_i)}$ in $\Sigma (L,L_{n_i})$ is less than $\epsilon$. 
On the other hand, if for some $n_i$ the angles corresponding to the lines in  the star $\Sigma (L,L_{n_{i}})$ determine a dense set in $[0,2\pi]$, we obtain immediately the existence of a sequence 
$\{\Delta_n\}\subset \Sigma (L,L_{n_i})$ of axes of symmetry of $K$ such that $\Delta_n \rightarrow \Delta$. 

In particular, we have proven the case  $k=2$ of the theorem. Since every line $L\subset V$, through $o$, is an axis of symmetry of the body $K$, by Theorem \ref{grandota} we conclude that $K$ is a 2-body of revolution. 


 We consider the case in which there is no plane $H\subset V$ containing an infinite subsequence of $\{L_n\}$. Then, letting $\Pi_n$ be the plane $span\{L,L_{n}\}$, we have that the sequence of planes $\{\Pi_n\}$ is infinite, hence there exists a limiting plane $\Pi\subset V$ for a subsequence, which we will still denote as $\Pi_{n}$. If for some integer $n_0$, the angles corresponding to the lines in $\Sigma (L,L_{n_0})$ are a dense set in $[0,2\pi]$, we have that for each line $\Delta \subset span\{L,L_{n_0}\}$ there exists a sequence $\{\Delta_n\}\subset \Sigma (L,L_{n_0})$ of axes of symmetry of $K$, such that $\Delta_n \rightarrow \Delta$. Thus,  by Lemma \ref{contento} [III], $\Delta$ is an axis of symmetry of the body $K$, and we obtain that all lines in $\Pi$ are axes of symmetry of $K$. On the other hand, if we do not have a dense set of angles, then every  $\Sigma (L,L_n)$  is a $f(n)$-star of lines,  for some integer $f(n)$. Since the sequence of planes $\{\Pi_n\}$ is infinite, the sequence $\{f(n)\}$ also is infinite. In virtue that $L_n \rightarrow L$, it follows that $f(n) \rightarrow \infty$. 
Let  $\Delta \subset \Pi$ be  a line, $x\in \Delta$. Again, we will show that there exists a sequence $\{\Delta_n\}$ of axes of symmetry of $K$ such that $\Delta_n \rightarrow \Delta$. Since 
$\Pi_n \rightarrow \Pi$, there exists a sequence of lines $\{\Gamma_n \}$  such that $\Gamma_n \subset \Pi_ n $ and $\Gamma_n \rightarrow \Delta$, i.e.,  given $\epsilon > 0$, there exists $N_1$ such that if  
$n> N_1$, then  $\delta_p(\Gamma_n,\Delta)< \epsilon/2$. On the other hand, since $f(n) \rightarrow \infty$, given $\epsilon > 0$, there exists $N_2$ such that $n> N_2$, there exists a line of $\Sigma (L,L_n)$, say  
$\Delta_{k(n)}$, such that $\delta_p(\Delta_{k(n)}, \Gamma_n)< \epsilon/2$. Taking  
$n> \max \{N_1,N_2\}$, then 
$\delta_p(\Delta_{k(n)}, \Delta)\leq \delta_p(\Delta_{k(n)}, \Gamma_n)+ \delta_p( \Gamma_n, \Delta)< \epsilon/2 + \epsilon/2=\epsilon$, i.e., $\Delta_{k(n)} \rightarrow \Delta$.

In the previous paragraph, we have proven that there exists a plane $\Pi\subset V$ such that every line in $\Pi$, through $o$, is an axis of symmetry of $K$. Since the case $k=2$ has already been established, let  $3\leq k \leq d-1$. Since $k>2$, there exists an integer $n_1$ such that $L_{n_1}$ is not contained in $\Pi$ and, by condition (iii) of Theorem \ref{dream}, $L_{n_1}$ is not perpendicular to $\Pi$. It follows from Lemma \ref{kiara} that every line in 
$span\{\Pi, L_{n_1}\}$, passing through $o$, is an axis of symmetry of $K$. If $d=3$, then $k=3$ and $\Et=span\{\Pi, L_{n_1}\}$ and, consequently, $K$ is a sphere by the property (A) in the introduction.  
If $d>3$ and $k=3$, then, by Theorem \ref{grandota}, $K$ is a 3-body of revolution. 
If $k>3$, using Lemma \ref{kiara} repeatedly, we can assume the existence of axes of symmetry $\{L_{n_1}, L_{n_2},\ldots,L_{n_{(k-2)}}\} \subset \{L_n\}$ of $K$ such that every line $L'\subset V=span\{\Pi, L_{n_1}, L_{n_2},\ldots,L_{n_{(k-2)}}\}$ through $o$ is an axis of symmetry of $K$. Thus, if $k\leq d-1$, then $K$ is a $k$-body of revolution by Theorem \ref{grandota}. If $k=d$, then $K$ is a sphere and the proof of Theorem \ref{dream} is now complete. $\Box$

\section{The proofs of Theorems \ref{fantasia} and \ref{brasil}}

\textit{Proof of Theorem} \ref{fantasia}.
In virtue of the hypothesis, there exists an integer $k$, $1\leq k \leq d-2$, and a $k$-flat 
$\Gamma$ such that the hyperplanes of symmetry $\Pi_i$ of $K$ satisfy (\ref{miami}). Let $\Delta$ be a 
$(d-k)$-flat orthogonal to $\Gamma$ and such that $\Delta \cap  K\not= \emptyset$. By (\ref{miami}), $\Delta$ is orthogonal to $\Pi_i, i=1,2,\ldots$. Consequently, $\Delta \cap \Pi_i$  is a $(d-k-1)$-flat of symmetry of $\Delta \cap K$. We claim that for  $1\leq t \leq d-k-1$, there is no $t$-flat $\Psi \subset \Delta$ and no subsequence $\Delta \cap \Pi_{i_j}$ of $\Delta \cap \Pi_i$ such that $ \Psi \subset  \Pi_{i_j}, j=1,2,\ldots$ Otherwise, we would have that $\aff\{ \Gamma, \Psi\}\subset \Pi_{i_j}$ and since  $\dim (\aff\{ \Gamma, \Psi\})>k$ we would contradict the maximality of $k$. By Proposition \ref{frances}, $\Delta \cap K$ is a sphere with center at $\Gamma$, i.e., $K$ is a $(d-k)$-body of revolution. 
 $\Box$    
 
\textit{Proof of Theorem} \ref{brasil}.  
Let $\Pi$ be a hyperplane, and $p\in \Pi$. We will show that for every hyperplane $\Gamma$ with $p\in \Gamma$ and $\Gamma\not=\Pi$,  the sections $\Pi \cap K $ and $\Gamma \cap K$ are congruent, i.e., there exists an orthogonal transformation 
$\Omega: \Ed \rightarrow \Ed$ such  that $\Omega(\Pi \cap K) =\Gamma \cap K$. We denote by $\Delta$ the subset  of $\mathbb{S}^{d-1}$ defined as follows:  $u\in \Delta$ if there exists a hyperplane $\Sigma$, $p\in \Sigma$, whose corresponding unit normal vector is $u$ and $\Sigma \cap K$ and $\Pi \cap K$ are congruent. We are going to show that $\Delta =\mathbb{S}^{d-1}$. We denote by $w$ the unit normal vector of $\Pi$. Let $L\subset \Pi$ be a  $(d-2)$-plane, $p\in L$. In virtue that $L$ is a $(d-2)$-axis of rotation of order $k$, there  are hyperplanes 
$\Sigma_1$, $\Sigma_2$ with unit normal vectors $v_1,v_2$, respectively, such that $\Pi \cap K$ can be obtained from $\Sigma_1 \cap K$ and $\Sigma_2 \cap K$ after applying two rotations, both with axis $L$, one by an angle $2\pi/k$ and the other by an angle $2(k-1)\pi/k$. Since the boundary of $K$ is a continuous surface, varying $L$, always  contained in  
$\Pi$, we have that the collection of unit normal vector of the hyperplanes $\Sigma$ is a  $(d-2)$-sphere $\tau_{w}\subset \mathbb{S}^{d-1}$ with center at $w$ and each vector $v\in \tau_{w}$ makes an angle $2\pi/k$ with $w$. It is clear that $\tau_{w}\subset \Delta$. Applying the same argument it follows that if $v$ is in $\Delta$, then $\tau_v \subset \Delta$. Thus 
\begin{eqnarray}\label{mole}
\bigcup_{v\in \tau_{w}} \tau_{v}\subset \Delta.
\end{eqnarray}
We denote by $B_w$ the interior of the  spherical cap with boundary $\tau_{w}$ in $\mathbb{S}^{d-1}$. In virtue that 
\[
B_w=\bigcup_{v\in \tau_{w}} \tau_{v},
\]
we get from (\ref{mole}) that $B_w \subset \Delta$. Finally, we observe that the collection of sets $\{B_v:v\in \tau_w\}$ covers the sphere. Hence $\Delta = \mathbb{S}^{d-1}$, i.e., any two sections of $K$ passing through $p$ are congruent.  By the S\"uss-Schneider  Theorem (see \cite{schneider, suss}), we conclude that $K$ is a sphere . $\Box$

\begin{remark}
Observe that a $(d-2)$-flat of symmetry of a convex body $K$ is, at the same time, a $(d-2)$-axis of rotation of order 2 of $K$. Therefore Theorem \ref{brasil}, for the case $k=d-2$, can be considered as a generalization of Theorem \ref{grandota} for $(d-2)$-axis of rotation of order 2. 
\end{remark}

We conclude the paper by stating a conjecture, which would reduce the hypothesis of Theorem \ref{brasil} to a countable collection of axes of rotation with not necessarily the same order.

\begin{conjecture}
Let $K\subset \Ed$ be a convex body, $d\geq 3$. Suppose that there exists a sequence $\{\Gamma_n\}\subset \Ed $ of $(d-2)$-axes of rotation of $K$ of order $\phi(\Gamma_n)$, such that the sequence $\{\phi(\Gamma_n)\}$ has an infinite number of terms different from 2. Then $K$ is a sphere.
\end{conjecture}

\end{document}